\documentclass[12pt]{amsart}
\usepackage{geometry}
\theoremstyle{theorem}
\newtheorem{theorem}{Theorem}[section]
\newtheorem{cor}[theorem]{Corollary}
\newtheorem{prop}[theorem]{Proposition}

\theoremstyle{remark}
\newtheorem{remark}{Remark}[section]
\newtheorem{definition}[remark]{Definition}
\newtheorem{example}[remark]{Example}
\newtheorem{algorithm}[remark]{Algorithm}

\newcommand{\polya}{P\"olya}

\newcommand{\rank}{\text{Rank}}
\newcommand{\RR}{\mathbb{R}}
\newcommand{\minimize}{\text{Minimize}}
\newcommand{\norm}[1]{||#1||}
\newcommand{\rh}{{\hat{r}}}
\newcommand{\ch}{{\hat{c}}}
\newcommand{\rhh}{{\overline{r}}}
\newcommand{\chh}{{\overline{c}}}

\begin{document}
\bibliographystyle{plain}

\title{Rank of Submatrices of the Pascal Matrix}
\author{Scott Kersey}
\address{Georgia Southern University, USA}
\email{scott.kersey@gmail.com}

\maketitle

\begin{abstract}
In a previous paper, we derived necessary and sufficient conditions for
the invertibility of square submatrices of the Pascal upper triangular matrix.
To do so, we established a connection with the two-point Birkhoff 
interpolation problem.
In this paper, we extend this result by
deriving a formula for the rank of submatrices of the Pascal matrix.
Our formula works for both square and non-square submatrices.
We also provide bases for the row and column spaces of these submatrices.
Further, we apply our result to one-point lacunary polynomial approximation.
\end{abstract}

Paper appears in: Journal of Mathematical Sciences: Advances and Applications, 42, 1--12 (2016).

\section{Introduction}

Pascal's triangle can be represented by the infinite upper triangular matrix
$$
T 
=\begin{bmatrix}
\binom{0}{0}&\binom{1}{0}&\binom{2}{0}&\cdots  \\[3pt]
\binom{0}{1}&\binom{1}{1}&\binom{2}{1}&\cdots  \\[3pt]
\binom{0}{2}&\binom{1}{2}&\binom{2}{2}&\cdots  \\[3pt]
\vdots & \vdots & \vdots & \ddots  \\
\end{bmatrix}
=\begin{bmatrix}
1 & 1 & 1 & \cdots  \\[3pt]
0 & 1 & 2 & \cdots  \\[3pt]
0 & 0 & 1 & \cdots  \\[3pt]
\vdots & \vdots & \vdots & \ddots  \\
\end{bmatrix}
$$
with $\binom{j}{i}:=0$ if $i>j$.
Submatrices are of the form
$$
T_{r,c} = 
\begin{bmatrix}
\binom{c_0}{r_0} & \binom{c_1}{r_0} & \cdots & \binom{c_n}{r_0} \\
\vdots & \vdots & \ddots & \vdots \\
\binom{c_0}{r_m} & \binom{c_1}{r_m} & \cdots & \binom{c_n}{r_m}
\end{bmatrix}
$$
for some \emph {selections} $r = [r_0, \ldots, r_m]$ and 
$c = [c_0, \ldots, c_n]$ of the rows and columns of $T$, respectively.
For example, 
$$
T_{[1,3,4],[0,4,5,7]} = 
\begin{bmatrix}
\binom{0}{1} & \binom{4}{1} & \binom{5}{1} & \binom{7}{1} \\[3pt]
\binom{0}{3} & \binom{4}{3} & \binom{5}{3} & \binom{7}{3} \\[3pt]
\binom{0}{4} & \binom{4}{4} & \binom{5}{4} & \binom{7}{4}
\end{bmatrix}
=
\begin{bmatrix} 0&4&5&7\\[3pt]0&4&10&35\\[3pt]0&1&5&35 \end{bmatrix}.
$$
While the rank of this matrix is $3$, it is not obvious to see.
The main goal in this paper is to provide a formula for determining 
the rank of such matrices,
and secondly to provide bases for the row and column spaces.
Later in the paper, we apply our results to a problem in 
polynomial approximation.

The results in this paper are a generalization of the main result in \cite{K16},
stated below as in Theorem \ref{t1}.
This theorem provides necessary and sufficient conditions for the invertibility 
of square submatrices of the Pascal matrix.
In the proof of that result, we showed that
the invertibility of square submatrices of $T$ is equivalent
to the unique solvability of a two-point Birkhoff interpolation problem.
This Birkhoff interpolation problem has been studied in \cite{B1906,F69,P31,W35},
and generalized to \emph{lacunary interpolation} (see \cite{P09,P11}).

In the present paper, we show how to determine the rank of submatrices of the upper 
triangular Pascal matrix.
Our result applies to both square and non-square submatrices.
Our main results are stated in Theorem \ref{t2} and Theorem \ref{t3}.
In Algorithm \ref{alg1}, we give an algorithm that demonstrates how to compute 
linearly independent rows and columns.
In the final section of this paper,
we apply our results to a problem in lacunary polynomial approximation.

\section{Rank of Submatrices of the Pascal Matrix}

In \cite{K16}, the following result was derived for square
submatrices of the Pascal triangle.

\begin{theorem}[{\cite[Theorem 1.1]{K16}}]
\label{t1}
Let $r = [r_0, \ldots, r_m]$ and $c = [c_0, \ldots, c_m]$ 
be indices to the rows and columns of the
square submatrix $T_{r,c}$ of the Pascal upper triangular matrix.
Then, $T_{r,c}$ is invertible iff 
the following equivalent conditions hold:
\begin{itemize}
\item $r \leq c$ (i.e., $r_i\leq c_i$ for all $i$).
\item There is no zero diagonal entry.
\end{itemize}
\end{theorem}


Our goal is to generalize this theorem by finding the rank of 
arbitrary submatrices, square or non-square.
These submatrices are defined by sequences of rows
$r = [r_0, \ldots, r_m]$ and columns $c = [c_0, \ldots, c_n]$ of $T$.
In general, $m\neq n$.
To prove our results, we focus on the condition $r_i \leq c_i$,
and to this end make the following definition.

\begin{definition}
\label{d1}
We say $\{\rh,\ch\}$ is an \emph{ordered sub-pair} 
for $\{r,c\}$ of \emph{length} $p+1$ if $\rh = [\rh_0, \ldots, \rh_p]$
is a subsequence of $r= [r_0, \ldots, r_m]$,
$\ch = [\ch_0, \ldots, \ch_p]$ is a subsequence of $c=[c_0, \ldots c_n]$,
and $\rh \leq \ch$ (i.e., $\rh_i \leq \ch_i$ for $i=0, \ldots, p$).
If $\rh=\ch=\emptyset$, then $p+1=0$.
We say $\{\rh,\ch\}$ is \emph{maximal}
if there is no ordered sup-pair of length greater than $p+1$.
\end{definition}

For example, consider $r=[2,7,11,14,17,20]$ and $c=[0,4,9,10,15]$.
Then,
$\{\rh,\ch\} = \{[2,11],[4,15]\}$ is an ordered sub-pair of length $2$.
But this sub-pair is not maximal since the
ordered sub-pair $\{[2,7,14],[4,9,15]\}$ is of length $3$,
which, as it turns out, is maximal.
Note that $\{[2,7,11],[9,10,15]\}$ is also a maximal sub-pair,
hence these need not be unique.

The sub-pairs provided in Definition \ref{d1} are used to construct
invertible and full-rank submatrices of $T_{r,c}$,
as we summarize in Theorem \ref{t2}.

\begin{theorem}
\label{t2}
Let $r = [r_0, \ldots, r_m]$ and $c = [c_0, \ldots, c_n]$ 
be indices to the rows and columns of the Pascal upper triangular matrix.
Suppose $\{\rh,\ch\}$ is an ordered sup-pair of $\{r,c\}$,
with $\rh = [\rh_0, \ldots, \rh_p]$ and $\ch = [\ch_0, \ldots, \ch_p]$.
Then,
\begin{itemize}
\item $T_{\rh,\ch}$ is a $(p+1)\times(p+1)$
invertible submatrix of $T_{r,c}$.
\item $T_{\rh,c}$ is a $(p+1)\times(n+1)$ submatrix of full row rank.
\item $T_{r,\ch}$ is a $(m+1)\times(p+1)$ submatrix of full column rank.
\end{itemize}
If $\{\rh,\ch\}$ is maximal,
\begin{itemize}
\item The rank of $T_{r,c}$ is $p+1$.
\item The columns of $T_{r,\ch}$ span the column space of $T_{r,c}$.
\item The rows of $T_{\rh,c}$ span the row space of $T_{r,c}$.
\end{itemize}
\end{theorem}

\begin{proof}
The dimensions of these submatrices follow from
$\#r=m+1$, $\#c=n+1$ and $\#\rh=\#\ch=p+1$.
Note that $p \leq \min\{m,n\}$.
Since $\{\rh,\ch\}$ is ordered, $\rh \leq \ch$.
By Theorem \ref{t1}, $T_{\rh,\ch}$ is invertible.
The rank of $T_{\rh,c}$ is at least $p+1$ because it 
contains the invertible matrix $T_{\rh,\ch}$ as a submatrix.
Since it is of dimension $(p+1)\times(n+1)$, 
it necessarily has a full row rank.
Likewise, $T_{r,\ch}$ has full column rank.

For the second part, suppose $\{\rh,\ch\}$ is maximal.
Since $T_{\rh,\ch}$ is a submatrix of $T_{r,c}$,
$$
\rank\big(T_{r,c}\big) \geq \rank\big(T_{\rh,\ch}\big) = p+1.
$$
Suppose $\rank\big(T_{r,c}\big) > p+1$.
Then, there exists rows
$\rhh = [\rhh_0, \ldots, \rhh_{p+1}]$
and columns
$\chh = [\chh_0, \ldots, \chh_{p+1}]$
of $T_{r,c}$ such that
$\rank\big(T_{\rhh,\chh}\big) = p+2$.
By Theorem \ref{t1}, $\rhh\leq\chh$,
and so $\{\rhh,\chh\}$ is an ordered sup-pair of length $p+2$.
But this contradicts the assumption that $\{\rh,\ch\}$ is maximal of order $p+1$.
Therefore, 
$\rank\big(T_{r,c}\big) \not> p+1$,
and so
$\rank\big(T_{r,c}\big) = p+1$.

Finally, if $\{\rh,\ch\}$ is maximal,
$\rank(T_{\rh,c}) = \rank(T_{r,c}) = p+1$,
and so $T_{\rh,c}$ spans the row space of $T_{r,c}$.
Likewise, $T_{r,\ch}$ spans the column space of $T_{r,c}$.
\end{proof}

\section{Computing the Rank of Submatrices of the Pascal Matrix}

In this section we show how to compute a maximal ordered sub-pair 
$\{\rh,\ch\}$ of $\{r,c\}$.
In Algorithm \ref{alg1}, we actually compute the indices $\alpha$ and $\beta$
of a maximal pair $\{\rh,\ch\}$,
such that $\rh = r_\alpha = [r_{\alpha_0}, \ldots, r_{\alpha_p}]$
and $\ch = c_\alpha = [c_{\alpha_0}, \ldots, c_{\alpha_p}]$.

\begin{algorithm}
\label{alg1}
Compute indices $\{\alpha,\beta\}$ to a maximal ordered sub-pair of $\{r,c\}$.
\begin{itemize}
\item Input $r = [r_0, \ldots, r_m]$ and $c=[c_0, \ldots, c_n]$.
\item If $r_0>c_n$, set $\alpha=\beta=\emptyset$, go to output.
\item Otherwise, set $M := \max\{i : r_i \leq c_n\}$.
\item Set $\beta_0 := \min \{k : r_0 \leq c_k\}.$
\item For $i=1, \ldots, M$, set 
  $\beta_i := \max\Big\{\min \{k : r_i \leq c_k\}, \ \beta_{i-1}+1\big) \Big\}.$
\item Set $p := \max\{i : \beta_i \leq n\}$.
\item Set $\alpha = \{0, \ldots, p\}$.
\item Set $\beta = \{\beta_0, \ldots, \beta_p\}$.
\item Output $\alpha$, $\beta$.
\end{itemize}
\end{algorithm}

The strategy of our algorithm is to pair each
$r_{\alpha_i}$ with the first $c_{\beta_i}$ such that
$r_{\alpha_i} \leq c_{\beta_i} > c_{\beta_{i-1}}$.
In this \emph{first occurrence} strategy, $\alpha = [0, \ldots, p]$.
Therefore, $\rh = r_\alpha = [r_0, \ldots, r_p]$.

\begin{theorem}
\label{t3}
Let $r = [r_0, \ldots, r_m]$ and $c = [c_0, \ldots, c_n]$ 
be indices to the rows and columns of the
rectangular submatrix $T_{r,c}$ of the Pascal upper triangular matrix.
Let $\{\alpha,\beta\}$ be determined by Algorithm \ref{alg1}.
If $\alpha=\beta=\emptyset$,
then $T_{r,c}$ is a matrix of zeros of rank $0$.
Otherwise, let $\rh = r_\alpha$ and $\ch=c_\beta$.
Then, $\{\rh,\ch\}$ is a maximal ordered sub-pair of $\{r,c\}$,
and the rank of the submatrices is determined by Theorem \ref{t2}.
\end{theorem}

\begin{proof}
Suppose $\alpha=\beta=\emptyset$.
By Algorithm \ref{alg1}, this occurs when $r_0 > c_n$.
In this case, $r_i > c_j$ for $0\leq i \leq m$ and $0\leq j \leq n$,
and so the elements $\binom{c_j}{r_i}$ of $T_{r,c}$ are all zero.
Hence, $T_{r,c}$ is the zero matrix, of rank $0$.

Suppose that $\alpha$ and $\beta$ are non-empty.
This occurs when $r_0 \leq c_n$.
Since $\alpha = [0, \ldots, p]$ and $\beta_i > \beta_{i-1}$,
both $\alpha$ and $\beta$ are strictly increasing sequences.
Since $r$ and $c$ are also strictly increasing, 
it follows that $\rh=r_\alpha$ and $\ch=c_\beta$ are both strictly increasing.
In the algorithm we choose $\beta_i$ to be the first occurrence such
that $\alpha_i \leq \beta_i$ and $\beta_i > \beta_{i-1}$.
Therefore, the $\beta_i$ are incremented by the least amount possible,
and so $p$ is the largest index for an ordered sub-pair of $\{r,c\}$.
Hence, $\{\rh,\ch\}$ is maximal.
\end{proof}

In the next result, we construct index sets of the same rank as
$T_{\rh,\ch}$ and $T_{r,c}$ with a minimal number of nonzero entries.

\begin{cor}
\label{c1}
Let $\{\rh,\ch\}$ be an ordered sub-pair of $\{r,c\}$.
Let $\alpha$ and $\beta$ be the corresponding index sets
such that $\rh=r_\alpha$ and $\ch = c_\beta$.
Let $I_{\rh,\ch}$ be the $(m+1)\times(n+1)$ matrix with 
$I_{\rh,\ch}(\alpha_i,\beta_i) = 1$, and all other entries zero.
Then, $\rank(T_{\rh,\ch}) = \rank(I_{\rh,\ch})$.
If $\{\rh,\ch\}$ is maximal, then $\rank(T_{r,c}) = \rank(I_{\rh,\ch})$.
\end{cor}

\begin{proof}
By Theorem \ref{t2}, the rank of $T_{\rh,\ch}$ is $p+1$.
This is the number of elements in $\rh$ and $\ch$,
which equals the number of elements in $\alpha$ and $\beta$.
The matrix $I_{\rh,\ch}$ has zeros for all entries except
for entry $(\alpha_i,\beta_i)$ for $i=0, \ldots, p$.
But, since $\alpha$ and $\beta$ are strictly increasing,
these non-zero entries are in different rows and columns.
Hence, the dimension of the row and column spaces of $I_{\rh,\ch}$ are
necessarily equal to the number of non-zero elements,
which is the length of $\alpha$ or $\beta$.
Therefore, the rank of $I_{\rh,\ch}$ is $p+1$,
which equals the rank of $T_{\rh,\ch}$.

In the case that $\{\rh,\ch\}$ is maximal,
we have by Theorem \ref{t2} that the rank of $T_{r,c}$
equals the rank of $T_{\rh,\ch}$, which equals the rank of $I_{\rh,\ch}$.
\end{proof}

We conclude this section with a constructive example.
\begin{example}
\label{ex1}
Let $r = [r_0, \ldots, r_5] = [2,7,11,14,17,20]$ and 
$c=[c_0, \ldots, c_4] = [0,4,9,10,15]$.
These are the rows and columns for the following submatrix of the Pascal matrix
$$
T_{r,c} = T_{[2,7,11,14,17,20],[0,4,9,10,15]} = 
\begin{bmatrix} 0&6&36&45&105\\ 0&0&36&120&6435\\ 
         0&0&0&0&1365\\ 0&0&0&0&15\\ 0&0&0&0&0\\ 0&0&0&0&0 \end{bmatrix}.
$$
Then, $m=5$ and $n=4$, and
$$
M = \max_i \{r_i \leq c_4\} = \max_i \{r_i \leq 15\} =  3
$$
since $14\leq 15$ but $17\not\leq 15$.
Then,
\begin{align*}
\beta_0 &= \{\min_k \big(r_0 \leq c_k)\}
     = \{\min_k \big(2 \leq c_k)\}
     = 1 \ \text{since} \ c_0<2\leq c_1, \\
\beta_1 &= \max\Big\{\min_k \big(r_1 \leq c_k), \ \beta_{0}+1 \Big\}
     = \max\Big\{\min_k \big(7 \leq c_k), \ 1+1 \Big\}
     = \max\Big\{2, \ 2 \Big\} = 2, \\
\beta_2 &= \max\Big\{\min_k \big(r_2 \leq c_k), \ \beta_{1}+1 \Big\}
     = \max\Big\{\min_k \big(11 \leq c_k), \ 2+1 \Big\}
     = \max\Big\{4, \ 3 \Big\} = 4, \\
\beta_3 &= \max\Big\{\min_k \big(r_3 \leq c_k), \ \beta_{2}+1 \Big\}
     = \max\Big\{\min_k \big(14 \leq c_k), \ 4+1 \Big\}
     = \max\Big\{4, \ 5 \Big\} = 5,
\end{align*}
and so
$$\beta = [\beta_i  : \beta_i \leq n] = [\beta_i  : \beta_i \leq 4] = 
  [\beta_0, \beta_1, \beta_2] = [1,2,4].$$
Hence, $p=2$.
Then,
$\rh = r_{[0,1,2]} = [2,7,11]$ and $\ch = c_{[1,2,4]} = [4,9,15]$.
Thus, we compute
$$
T_{\rh,\ch} = 
T_{[2,7,11],[4,9,15]} = 
\begin{bmatrix} 6&36&105\\ 0&36&6435\\ 0&0&1365 \end{bmatrix}.
$$
By Theorem \ref{t2}, 
$$\rank(T_{r,c}) = \rank(T_{\rh,\ch}) = p+1 = 3$$
and $T_{\rh,\ch}$ is invertible.
Also by this theorem, the matrix
$$T_{\rh,c} = T_{[r_0,r_1,r_2],c} = T_{[2,7,11],[0,4,9,10,15]}$$
has linearly independent rows (full row rank),
and the matrix
$$T_{r,\ch} = T_{r,[c_1,c_2,c_4]} = T_{[2,7,11,14,17,20],[4,9,15]}$$
has linear independent columns (full column rank).
The index matrix $I_{\rh,\ch}$ defined in 
Corollary \ref{c1} is nonzero only at entries $(\alpha_i,\beta_i)$.
These are $(0,1)$, $(1,2)$ and $(2,4)$.
Hence,
$$
I_{\rh,\ch} = \begin{bmatrix} 0&1&0&0&0\\ 0&0&1&0&0\\ 
         0&0&0&0&1\\ 0&0&0&0&0\\ 0&0&0&0&0\\ 0&0&0&0&0 \end{bmatrix},
$$
which is rank $3$.
\end{example}

\section{Application to Polynomial Approximation}

In \cite{K16}, we established a connection between submatrices of
the Pascal upper triangular matrix and polynomial interpolation.
Let $r = [r_0, \ldots, r_m]$ and $c = [c_0, \ldots, c_n]$.
Let 
$$
\Lambda_{x,r} = 
\Big[\frac{\delta_xD^{r_0}}{r_0!}, \ldots, \frac{\delta_x D^{r_m}}{r_m!}\Big]
$$
with
$$
\delta_{x} D^{r_i} : f \mapsto f^{(r_i)}(x),
$$
and let
$$
V_c = \Big[(\cdot)^{c_0}, \ldots, (\cdot)^{c_n}\Big]
$$
be the power basis with powers $c_i$ for $i=0, \ldots, n$.
Then, the matrix
$$
\Lambda_{x,r}^TV_c = \Big [ \dfrac{\delta_xD^{r_i} (\cdot)^{c_j}}{c_j!}  \Big]
$$
is a kind of \emph{generalized Vandermonde} for the one-point lacunary
polynomial interpolation problem.
This kind of matrix has been considered in \cite{P09,P11}, but for square
matrices.
In \cite{K16} we showed that this generalized Vandermonde,
for the case that $x=1$, is related to the Pascal matrix as follows:

\begin{prop}
\label{p1}
Let $r = [r_0, \ldots, r_m]$ and $c = [c_0, \ldots, c_n]$ be selections
of the rows and columns of the Pascal upper triangular matrix $T$.
Then, $T_{r,c} = \Lambda_{1,r}^TV_c$.
\end{prop}

Hence, by Theorem \ref{t3}, we have the following:

\begin{cor}
\label{c2}
Let $r = [r_0, \ldots, r_m]$ and $c = [c_0, \ldots, c_n]$.
Assume $r_0 \leq c_n$.
Let $\{\rh,\ch\}$ be an ordered sub-pair of $\{r,c\}$.
Then, $\Lambda_{1,r}^TV_{\ch}$ has full column rank.
\end{cor}

Now we will apply our result to polynomial approximation.
Let $\{\rh,\ch\}$ be an ordered sub-pair of $\{r,c\}$ of length $p+1$.
We may assume it is maximal, but it doesn't need to be.
Let
$$
f(x) = \sum_{j=0}^p b_j x^{\ch_j}.
$$
Assume we are given some data, $y = [y_0, \ldots, y_m]$.
Since typically $m>p$, we cannot expect to interpolate.
However, by Corollary \ref{c2},
$\Lambda_{1,r}^TV_{\ch}$ has full column rank,
and so we can easily solve the least squares problem
$$
\minimize \Big\{\sum_{i=0}^m \norm{D^{r_i} f(1) - y_i}_2^2, \ b\in\RR^{p+1} \Big\}.
$$
The unique solution to this problem is
$$
b = A^+y = (A^TA)^{-1}A^Ty,
$$
where $A^+$ is the pseudo-inverse of $A:=\Lambda_{1,r}^TV_{\ch}$.

We conclude with an example.

\begin{example}
\label{ex2}
Let $r = [r_0, \ldots, r_5] = [2,7,11,14,17,20]$ and 
$c=[c_0, \ldots, c_4] = [0,4,9,10,15]$.
In Example \ref{ex1}, it was shown that
$\{\rh,\ch\} = \{[2,7,11],[4,9,15]\}$ is
a maximal ordered sub-pair of $\{r,c\}$.
Therefore, $\ch=[4,9,15]$ is the degree sequence of our lacunary polynomial, and
$$
A = \Lambda_{1,r}^TV_{\ch}
=
\begin{bmatrix}
6&36&105 \\
0&36&6435 \\
0&0&1365 \\
0&0&15 \\
0&0&0 \\
0&0&0
\end{bmatrix}.
$$
Suppose $y = [1,1,1,1,1,1]^T$.
Then, the coefficient sequence of the least squares solution is
$$
b = \begin{bmatrix}.7813 &  -.1046 & .0007\end{bmatrix}^T.
$$
Hence, the least squares lacunary polynomial is
$$
f(x) = .7813 x^4 - .1046 x^9 + .0007 x^{15}.
$$
\end{example}


\end{document}